\documentclass[11 pt]{article}

\usepackage[latin1]{inputenc}
\usepackage{amsmath,amsfonts,stmaryrd,amsthm,thmtools,hyperref,graphicx}

\usepackage{caption,subcaption}

\usepackage{fullpage}
\usepackage{authblk}

\newtheorem{thm}{Theorem}
\newtheorem{lemma}[thm]{Lemma}

\newtheorem{cor}[thm]{Corollary}
\theoremstyle{definition}

\newtheorem{rem}[thm]{Remark}

\newtheorem{condition}[thm]{Condition}

\numberwithin{definition}{section}
\numberwithin{proc}{section}
\numberwithin{equation}{section}
\numberwithin{condition}{section}
\numberwithin{condition}{section}
\numberwithin{prop}{section}
\numberwithin{thm}{section}
\numberwithin{lemma}{section}
\numberwithin{rem}{section}
\numberwithin{cla}{section}
\numberwithin{obs}{section}
\numberwithin{cor}{section}
\numberwithin{conjecture}{section}

\usepackage{enumerate}

\usepackage{xcolor}
\definecolor{webgreen}{rgb}{0,.5,0}
\definecolor{Maroon}{HTML}{800000}

\hypersetup{colorlinks, breaklinks, urlcolor=Maroon, linkcolor=Maroon, citecolor=webgreen} % Set link colors

\usepackage[numbers]{natbib}
\newcommand*{\doi}[1]{doi:\,\texttt{\href{http://dx.doi.org/#1}{#1}}}

\usepackage{stmaryrd}
\usepackage{bbold} % mathbb

\newcommand{\bE}{\mathbb{E}}

\newcommand\vbfd{{\vec{\mathbf{d}}}}

\newcommand\dsD{{\mathbb D}}
\newcommand\dsE{{\mathbb E}}

\newcommand\dsG{{\mathbb G}}

\newcommand\dsN{{\mathbb N}}

\newcommand\dsZ{{\mathbb Z}}

\usepackage{physics}
\usepackage{mathtools}
\newcommand{\Z}{{\mathbb Z}}
\newcommand{\N}{{\mathbb N}}

\newcommand{\da}{\downarrow}
\newcommand{\ua}{\uparrow}

\newcommand\cA{{\cal A}}

\newcommand\cE{{\cal E}}
\newcommand\cF{{\cal F}}

\newcommand\cH{{\cal H}}
\newcommand\cI{{\cal I}}

\newcommand\cL{{\cal L}}

\newcommand\cN{{\cal N}}

\newcommand\cP{{\cal P}}

\newcommand\cU{{\cal U}}
\newcommand\cV{{\cal V}}

\newcommand\cX{{\cal X}}

\newcommand{\E}[1]{{\mathbb E}\left[#1\right]}
\newcommand{\ee}[1]{{\mathbb E}[#1]}

\newcommand{\p}[1]{{\mathbb P}\left\{#1\right\}}
\newcommand{\pp}[1]{{\mathbb P}\{#1\}}

\newcommand\cond{\;\middle|\;}

\usepackage[authormarkup=none]{changes}

\definechangesauthor[name={Cai},color=red]{C}

\definechangesauthor[name={Guillem},color=blue]{G}

\DeclarePairedDelimiter{\floor}{\lfloor}{\rfloor}
\DeclarePairedDelimiter{\ceil}{\lceil}{\rceil}

\usepackage{xspace}

\newcommand{\tin}{\mathrm{in}}
\newcommand{\tout}{\mathrm{out}}

\newcommand\hnu{{\hat \nu}}

\DeclareMathOperator{\dist}{dist}

\newcommand\din{D_{{\tin}}}

\newcommand\dinp{\din^{+}}

\newcommand\dnin{(D_{n})_{\tin}}
\newcommand\dninp{\dnin^{+}}

\newcommand\dout{D_{{\tout}}}

\newcommand\doutm{\dout^{-}}
\newcommand\dnout{(D_{n})_{\tout}}

\newcommand\dnoutm{\dnout^{-}}

\newcommand\Qnda{Q_{n}^{\downarrow}}
\newcommand\Qnua{Q_{n}^{\uparrow}}

\newcommand\vecGn{\vec{\dsG}_{n}}

\newcommand\GW{\mathrm{GW}}

\newcommand\giant{\mathcal{G}_{n}}

\newcommand\NtoO[1]{\cN^{\le \omega}(#1)}

\newcommand\NtoOXpm{\NtoO{\cX^{\pm}}}

\newcommand\scc{\textsc{scc}\xspace}

\title{The giant component of the directed configuration model revisited}

\author[*]{Xing Shi Cai}
\author[**]{Guillem Perarnau}
\affil[*]{\small\it Uppsala University, Sweden. Email:~{\tt xingshi.cai@math.uu.se}.}
\affil[**]{\small\it Departament de Matem\`atiques. UPC. Email:~{\tt guillem.perarnau@upc.edu}.}

\begin{document}
\maketitle

\begin{abstract}
We prove a law of large numbers for the order and size of the largest strongly connected component in the directed configuration model. Our result extends previous work by Cooper and
Frieze~\cite{cooper2004}.
\end{abstract}

\section{Introduction and notations}

An \scc (strongly connected component) in a digraph (directed graph) is a maximal sub-digraph in
which there exists a directed path from every node to every other node.  In this short note, we analyse
the size of the giant component, i.e., the largest \scc, in the directed configuration model.  This
is a continuation of our previous work \cite{cai2020a}, which studied the diameter of the model.  

We briefly introduce the model and our assumptions. For further discussions and references, see
\cite{cai2020a}. Let \([n]\coloneqq \{1,\dots,n\}\) be a set of \(n\) nodes.  Let
\(\vbfd_n=((d^{-}_1,d_1^{+}),\dots, (d^{-}_n,d^{+}_n))\) be a bi-degree sequence with \(m_n\coloneqq \sum_{i\in
[n]} d^{+}_i = \sum_{i \in [n]}d^{-}_{i}\).  The directed configuration model, \(\vecGn\), is the
random directed multigraph on \([n]\) generated by giving \(d^{-}_{i}\) in half-edges (\emph{heads})
and \(d^{+}_{i}\) out half-edges (\emph{tails}) to node \(i\), and then pairing the heads and tails
uniformly at random.

Let \(D_{n}=(D_{n}^{-},D_{n}^{+})\) be the degrees (number of tails and heads) of a uniform random node.
Let \(n_{k,\ell}\) be the number of \((k,\ell)\) in \(\vbfd_{n}\).
Let \(\Delta_n= \max_{i\in[n]} \{d^{-}_i,d^{+}_i\}\).
Consider a sequence of bi-degree sequences \((\vbfd_n)_{n\geq 1}\). 
Throughout the paper, we will assume the following condition is satisfied,
\begin{condition}\label{cond:main}
    There exists a discrete probability distribution \(D=(D^{-},D^{+})\) on \(\dsZ_{\ge 0}^2\) with \(\lambda_{k,
    \ell}\coloneqq\p{D=(k,\ell)}\) such that 
    \begin{enumerate}[(i)]
        \item \(D_n\) converges to \(D\) in distribution: \(\lim_{n\to \infty} \frac{n_{k,\ell}}{n}  = \lambda_{k,\ell}\) for every \(k,\ell\in
                \dsZ_{\geq 0};\)
        \item \(D_n\) converges to \(D\) in expectation and the expectation is finite: 
            \begin{equation}
                \lim_{n\to \infty} \dsE[D^{-}_n] =\lim_{n\to \infty} \dsE[D^{+}_n] =
                \dsE[D^{-}]=\dsE[D^{+}]\eqqcolon\lambda \in (0, \infty);
            \end{equation}
        \item \(D_n\) converges to \(D\) in second moment and they are finite: for \(i,j\in \Z_{\geq 0}\), \(i+j=2\),
               \begin{equation}
               \lim_{n\to \infty} \dsE[(D^{-}_n)^i(D^{+}_n)^j] = \dsE[(D^{-})^i(D^{+})^j] < \infty
               \end{equation} 
                
    \end{enumerate}
\end{condition}

To state the main result, some parameters of \(D\) are needed. 
Let
\begin{equation}\label{eq:supercri}
\nu
\coloneqq
\frac{\dsE[D^{-}D^{+}]}{\lambda}
< \infty
,
\end{equation}
where the inequality follows from conditions (ii) and (iii).  Let 
\(    f(z,w)\coloneqq \sum_{i ,j \ge 0} \lambda_{i,j} z^{i} w^{j}\)
be the bivariate generating function of \(D\). Let \(s_{-}\) and \(s_{+}\) be the survival probabilities
of the branching processes with offspring distributions which have generating functions
\(\frac{1}{\lambda}\frac{\partial f}{\partial w}(z,1)\) and \(\frac{1}{\lambda}\frac{\partial
f}{\partial z}(1,w)\) respectively.  In other words, \(\rho_{-} \coloneqq 1-s_{-}\) and
\(\rho_{+} \coloneqq 1-s_{+}\) are, respectively, the smallest
positive solutions to the equations
\begin{equation}\label{QPICZ}
    z = \frac{1}{\lambda} \frac{\partial f}{\partial w} (z, 1),
    \qquad
    w = \frac{1}{\lambda} \frac{\partial f}{\partial z} (1, w).
\end{equation}

Let \(\giant\) be the largest \scc in \(\vecGn\). (If there is more than one such \scc, we choose
an arbitrary one among them as \(\giant\).)  Let \(v(\giant)\) be the number of nodes in \(\giant\).
Let \(e(\giant)\) be the number of edges in \(\giant\).
Our main result is the following theorem on \(\giant\):
\begin{thm}\label{thm:giant}
    Suppose that \((\vbfd_n)_{n \ge 1}\) satisfies~\autoref{cond:main}. 
    If \(\nu > 1\), then 
    \begin{gather}
        \frac{
            v(\giant)
        }{n}
        \to
        \eta
        <
        \infty
        ,
		\label{ONDOS}        
        \\
        \frac{
            e(\giant)
        }{n}
        \to
        \lambda s_{-} s_{+}
        <
        \infty
        ,
        \label{MYYET}
    \end{gather}
    in expectation, in second moment and in probability,
    where
    \begin{equation}
        \eta
        \coloneqq \sum_{i ,j \ge 0} 
        \lambda_{i,j} 
        (1-\rho_{-}^{i})
        (1-\rho_{+}^{j})
    = 1+f(\rho_{-},\rho_{+})- f(\rho_{-},1)-f(1,\rho_{+})
        .
        \label{FXADH}
    \end{equation}
    If \(\nu < 1\), then for all \(a_{n}\) with \(a_{n} \to \infty\)
    \begin{equation}\label{FSXMO}
        \frac{
            v(\giant)
        }{a_{n}}
        \to
        0
        ,
    \end{equation}
    in expectation and in probability.
\end{thm}

\begin{rem}
Under~\autoref{cond:main}, the probability that \(\vecGn\) is simple is
    bounded away from \(0\), see~\cite{blanchet2013,janson2009}. Thus~\autoref{thm:giant} holds for a uniform random \emph{simple} digraph with degree
    sequence \(\vbfd_{n}\).
\end{rem}

The two cases \(\nu < 1\) and \(\nu > 1\) are often referred to as \emph{subcritical} and
\emph{supercritical} regimes.  As shown in \cite{cai2020a}, in the supercritical case, \(s_{\pm} > 0\) and \(\eta > 0\). In other words,
whp (with high probability), the size of the largest \scc is \emph{bounded} in the first case and
\emph{linear} in the second one.

Equation \eqref{ONDOS} in \autoref{thm:giant} was first proved by Cooper and
Frieze~\cite{cooper2004} under stronger conditions including \(\ee{(D_{n}^{+})^{2} D_{n}^{-}}  = o(\Delta_{n})\),
        \(\ee{(D_{n}^{-})^{2} D_{n}^{+}}  = o(\Delta_{n})\) and \(\Delta_n= o(n^{1/12})\).
Graf~\cite[Theorem 4.1]{graf2016} extended the existence of a linear order \scc  provided that \(\ee{D_{n}^{+}D_{n}^{-}}\) converges uniformly and  \(\Delta_n =o(n^{1/4})\).  \autoref{cond:main} only implies that \(\Delta_{n} = o(\sqrt{n})\), see
\cite[Corollary~2.4]{cai2020a}.
In the subcritical case, the results in \cite{cooper2004,graf2016} only show that
whp the largest \scc has order \(O(\Delta_{n}^{2} \log{n})\) instead of \(O(1)\). 

The paper is organized as follows: In \autoref{sec:branching}, we study the probability of
certain events for branching processes. 
In \autoref{sec:coupl}, we recall a graph exploration process defined
in \cite{cai2020a} and extend it. \autoref{sec:tower} studies the probability that a
set of half-edges to reach a large number of other half-edges. \autoref{sec:L} shows that the number
of nodes which can reach and can be reached from many nodes is concentrated around its mean.
Then in \autoref{sec:final} we show that these nodes form the giant. Finally in \autoref{sec:eg} we
give an application of \autoref{thm:giant} to binomial random digraphs. 

\section{Branching processes}\label{sec:branching}

Let \(\xi\) be a random variable on \(\dsZ_{\ge 0}\) and let
\((\xi_{i,t})_{i\ge 1, t \ge 0}\) be iid (independent and identically distributed) copies of \(\xi\). Let \(h_{\xi}\) be the generating function of \(\xi\) and \(\nu_{\xi}\coloneqq h_{\xi}'(1)=\E{\xi}\).
Let  \((X_t)_{t \ge 0}\) be a  branching process with offspring
distribution \(\xi\).  If \(X_{t} > 0\) for all \(t\), then the branching process is said to \emph{survive}; otherwise, it
is said to \emph{become extinct}.  The following are well-known in the branching process theory (see,
e.g., \cite[Theorem 3.1]{vanderhofstad2020} and \cite[Theorem~I.10.3]{athreya1972}, respectively):
\begin{lemma}\label{thm:survive}
    Let \(\rho_{\xi}\) be the smallest nonnegative solution of \(z = h_{\xi}(z)\). The survival
    probability is
    \begin{equation}\label{JIMZD}
        s_{\xi} \coloneqq \pp{\cap_{t \ge 1} [X_{t} > 0]} = 1 - \rho_{\xi}.
    \end{equation}
    Moreover, \(s_{\xi} > 0\) if and only if \(\nu_{\xi} > 1\).
\end{lemma}
\begin{lemma}\label{thm:GW:limit}
    Assume that \(\nu_{\xi} \in (1,\infty)\).  Then there exists a sequence
    \((m_{\xi,t})_{t \ge 0}\) for which \(m_{\xi,t}^{1/t} \to \nu\), such that \(X_{t}/m_{\xi,t}
    \to W_{\xi}\),
    where \(W_{\xi}\) is a non-negative random variable for which \(\p{W_{\xi} = 0} = 1-s_{\xi}\) and
    which is continuously distributed on \((0,\infty)\).
\end{lemma}

The main result of this section is the following:

\begin{lemma}\label{BEANT}
    Let \((X_t)_{t \ge 0}\) be a branching process with offspring distribution \(\xi\) with
    \(\nu_{\xi} \in (1,\infty)\). Let
    \begin{equation}\label{VXTVB}
        T_{\omega} \coloneqq \inf \{t : X_{t} \ge \omega \}.
    \end{equation}
    Then
    for all \(\varepsilon > 0\) and as \(\omega \to \infty\),
    \begin{equation}\label{BWUZP}
        \pp{T_{\omega} \le (1+\varepsilon) \log_{\nu_{\xi}} \omega}
        \to
        s_{\xi}.
    \end{equation}
    %as \(\omega \to \infty\),
    %where \(s_{\xi}\) is the survival probability of \((X_{t})_{t \ge 0}\).
\end{lemma}

\begin{proof}
    Let    \(    t_{1} = \floor{(1+\varepsilon) \log_{\nu_\xi}\omega}+1\).
    It suffices to show that
    \(\pp{T_{\omega} > t_{1}} \to q_{\xi} \coloneqq 1 - s_{\xi}\). We split this probability into 
    \begin{equation}\label{CVDKA}
        \pp{T_{\omega} > t_{1}} 
        = 
        \pp{[T_{\omega} > t_{1}] \cap [X_{t_{1}} = 0]}
        +
        \pp{[T_{\omega} > t_{1}] \cap [X_{t_{1}} \in (0,\omega)]}
        \eqqcolon I_{1} + I_{2}.
    \end{equation}

    By Theorem~3.4 of \cite{cai2020a}, there exist constants \(C > 0\) and \(\hnu \in (0,1)\) (both depending only on \(\xi\))
    such that for all \(\varepsilon > 0\),
    \begin{equation}\label{KTMGN}
        I_{2}
        = \pp{\cap_{i=0}^{t_{1}} X_{i} \in (0,\omega)}
        %\le C \hnu^{t_{1} - m_{\xi,\omega}}
        \le C \hnu^{(1+\varepsilon) \log_{\nu_\xi} \omega - (1+o(1)) \log_{\nu_\xi} \omega-1}
        \le C \hnu^{(\varepsilon/2) \log_{\nu_\xi} \omega}
        = o(1)
        .
    \end{equation}

    Let \(Y_{t} = \sum_{i=0}^{t} X_{i}\).
    Let \(E\) denote the event that \((X_{t})_{t \ge 0}\) becomes extinct, i.e., \(X_{t} = 0\) for
    some \(t \in \dsN\).  If \(q_{\xi} = \p{E} = 0\), then \(I_{1} = 0\) and we are done.
    Thus we can assume that \(q_{\xi} > 0\). Then
    \begin{equation}\label{ZFDID}
        I_{1} 
        \le \p{\left[Y_{t_{1}} \le (1+t_{1}) \omega \right] \cap [X_{t_{1}} = 0]}
        \le \p{Y_{t_{1}} \le (1+t_{1}) \omega  \cond E} \pp{E}
        \to \p{E} = q_{\xi}
        ,
    \end{equation}
    since a branching process conditioned on becoming extinct
    has a finite total progeny.

    For a lower bound of \(I_{1}\), note that
    \(Y_{t} < \omega\) implies \(T_{\omega} > t\). Thus,
    \begin{equation}\label{BWKJZ}
        I_{1} \ge  \pp{[Y_{t_{1}} < \omega ] \cap [X_{t_{1}} = 0]}
        =  
        \pp{Y_{t_{1}} < \omega}
        -
        \pp{[Y_{t_{1}} < \omega ] \cap [X_{t_{1}} > 0]}
        .
    \end{equation}
    Note that
    \begin{equation}\label{GQTVR}
        \pp{Y_{t_{1}} < \omega}
        \ge
        \p{Y_{t_{1}} < \omega \cond E} \pp{E}
        \to 
        \pp{E}
        =
        q_{\xi}
        .
    \end{equation}
    By Theorem~6 of \cite{pakes1971}, there exists a sequence \((r_{t})_{t \ge 0}\) with
    \(r_{t}^{1/t} \to \nu_\xi\) such that for all \(x > 0\),
    \begin{equation}\label{BBXBA}
        \p{
            \frac{Y_{t_{1}}}{r_{t_{1}}} < x \cond X_{t_{1}} > 0
        }
        \to
        \p{Z_{\xi} < x \cond Z_{\xi} > 0}
        ,
    \end{equation}
    where \(Z_{\xi}\) is a non-negative random variable for which \(\p{Z_{\xi}=0}=q_{\xi}\) and
    which has continuous distribution on \((0,\infty)\). Therefore, for all \(\delta > 0\),
    \begin{equation}\label{DDJEM}
        %\begin{aligned}
        \p{Y_{t_{1}}  < \omega\cond X_{t_{1}} > 0}
        %&
        %=
        %\p{
        %    \frac{Y_{t_{1}}}{r_{t_{1}}}
        %    <
        %    \frac{\omega}{r_{t_{1}}}
        %    \cond
        %    X_{t_{1}} > 0
        %}
        %\\
        %&
        %\le
        %\p{
        %    \frac{Y_{t_{1}}}{r_{t_{1}}}
        %    <
        %    \omega^{-\varepsilon/2}
        %    \cond
        %    X_{t_{1}} > 0
        %}
        %\\
        %&
        \le
        \p{
            \frac{Y_{t_{1}}}{r_{t_{1}}}
            <
            \delta
            \cond
            X_{t_{1}} > 0
        }
        %\\
        %&
        \to
        \p{Z_{\xi} < \delta \cond Z_{\xi} > 0}
        ,
        %\end{aligned}
    \end{equation}
    as \(\omega \to \infty\).
    Since \(\delta\) is arbitrary, we have 
    \begin{equation}\label{XMHMW}
        \p{Y_{t_{1}}  < \omega\cond X_{t_{1}} > 0} \to 0.
    \end{equation}
    Putting \eqref{XMHMW} and \eqref{GQTVR} into \eqref{BWKJZ} gives the desired lower bound.
\end{proof}

\autoref{BEANT} can be generalized to multiple iid 
branching processes as follows:

\begin{cor}\label{ORGNW}
    Let \((X_{1,t})_{t \ge 0},\dots, (X_{x,t})_{t \ge 0}\) be \(x \in \dsN\) independent branching processes
    with offspring distribution \(\xi\).  Assume that \(\nu_{\xi} \in (1,\infty)\). Let
    \begin{equation}\label{XYVLM}
        T_{\omega}^{(x)} \coloneqq \inf \Big\{t : \sum_{i=1}^{x} X_{i, t} \ge \omega \Big\}.
    \end{equation}
    Then 
    for all \(\varepsilon > 0\) and as \(\omega \to \infty\),
    \begin{equation}\label{ZUEGE}
        \pp{T_{\omega}^{(x)} \le (1+\varepsilon) \log_{\nu_{\xi}} \omega}
        \to
        1-(1-s_{\xi})^{x}.
    \end{equation}
\end{cor}

\begin{proof}
    Let \(t_{1}=\floor{(1+\varepsilon) \log_{\nu_\xi}\omega}+1\).
    Let \(T_{i,\omega} = \inf\{t\geq 1:X_{i,t} \ge \omega\}\). By \autoref{BEANT}
    \begin{equation}\label{JIBNE}
        \pp{T_{\omega}^{(x)} > t_{1}}
        \le
        \pp{\cap_{i=1}^{x} [T_{i, \omega} > t_{1}]}
        =
        \prod_{i=1}^{x}
        \pp{T_{i, \omega} > t_{1}}
        \to
        (1-s_{\xi})^{x},
    \end{equation}
    and
    \begin{equation}\label{GHCZX}
        \pp{T_{\omega}^{(x)} > t_{1}}
        \ge
        \pp{\cap_{i=1}^{x} [T_{i, \frac{\omega}{x}} > t_{1}]}
        =
        \prod_{i=1}^{x}
        \pp{T_{i, \frac{\omega}{x}} > t_{1}}
        \to
        (1-s_{\xi})^{x}.\qedhere
    \end{equation}
\end{proof}

\section{Exploring the graph}\label{sec:coupl}

We extend the Breadth First Search (BFS) graph exploration process of \(\vecGn\) defined in \cite{cai2020a}.

For \(\cI \subseteq [n]\), let \(\cE^{\pm}(\cI)\) be the set of heads/tails incident to the nodes in
\(\cI\). Let \(\cE^{\pm}\coloneqq\cE^{\pm}([n])\). For \(\cX\subseteq
\cE^\pm\), let \(\cV(\cX)\) be the set of nodes incident to \(\cX\).
Let \(H\) be a partial pairing of half edges in \(\cE^{\pm}\).  Let \(\cP^{\pm}(H)\subseteq
\cE^{\pm}\) be the set of heads/tails which are paired in \(H\). Let \(\cV(H)=\cV(\cP^\pm(H))\).
Let \(\cF^{\pm}(H)\coloneqq\cE^{\pm}(\cV(H))\setminus\cP^{\pm}(H)\) be the unpaired heads/tails
which are incident to \(\cV(H)\).  Let \(E_{H}\) denote the event that \(H\) is part of \(\vecGn\).
We will explore the graph conditioning on \(E_{H}\).

We start from an arbitrary set \(\cX^+\) of \emph{unpaired} tails.  In this process,
we create random pairings of half-edges one by one and keep each half-edge in exactly one of the four states --- \emph{active, paired}, \emph{fatal} or \emph{undiscovered}. Let \(\cA_i^{\pm }\), \(\cP_i^{\pm}\), \(\cF_{i}^{\pm}\) and  \(\cU_i^{\pm }\) denote
the set of heads/tails in the four states respectively after the \(i\)-th pairing of
half-edges. Initially, let
\begin{equation}\label{eq:APUF}
    \cA_{0}^{+}=\cX^{+},\;
    \cA_{0}^{-}=\cE^{-}(\cV(\cX^+)) ,\;
    \cP_0^{\pm}=\cP^{\pm}(H),\;
    \cF_0^{\pm}=\cF^{\pm}(H),\;
    \cU_{0}^{\pm}=\cE^{\pm} \setminus (\cA^{\pm}_{0} \cup    \cP^{\pm}_{0} \cup \cF_0^{\pm})    .
\end{equation}
Then set \(i=1\) and proceed as follows: \leavevmode
\begin{enumerate}[\normalfont(i)]
    \item Let \(e_{i}^{+}\) be one of the tails which became active earliest in
        \(\cA_{i-1}^{+}\).
    \item Pair \(e^{+}_{i}\) with a head \(e^{-}_{i}\) chosen uniformly at random from \(\cE^-\setminus \cP^-_{i-1}\). Let \(\cP_{i}^{\pm} = \cP_{i-1}^{\pm} \cup \{e^{\pm}_{i}\}\).
    \item If \(e^{-}_{i} \in \cF_{i-1}^{-}\), then terminate; if \(e^{-}_{i} \in \cA_{i-1}^{-}\),
        then \(\cA_{i}^{\pm} = \cA_{i-1}^{\pm}\setminus \{e_i^\pm\}\); and if \(e_i^{-}\in
        \cU^{-}_{i-1}\), then \(\cA_{i}^{\pm} = (\cA_{i-1}^{\pm} \cup \cE^{\pm}(v_{i})) \setminus
        \{e^{\pm}_{i}\}\)  where \(v_{i}=\cV(e^{-}_i)\).
    \item If \(\cA_{i}^{+}\!\!=\emptyset\) terminate; otherwise, \(\cF_{i}^{\pm}\! =\! \cF_{i-1}^{\pm}\), \(\cU_i^{\pm}\! =\!\cE^{\pm}\setminus (\cA^{\pm}_{i} \cup    \cP^{\pm}_{i} \cup \cF_i^{\pm})\), \(i=i+1\) and go to (i).
\end{enumerate}

Let \(F_{\cX^{+}}(0)\) be a forest with \(|\cX^+|\) isolated nodes corresponding
to \(\cX^{+}\).  Given \(F_{\cX^{+}}(i-1)\), \(F_{\cX^{+}}(i)\) is constructed as follows:  if \(e_{i}^{-}
\in \cU_{i-1}^{-}\), then construct \(F_{\cX^{+}}(i)\) from  \(F_{\cX^{+}}(i-1)\) by adding \(\abs{\cE^{+}(v_{i})}\) child nodes to the node representing \(e_{i}^{+}\), each of which representing a
tail in \(\cE^{+}(v_{i})\); otherwise, let \(F_{\cX^{+}}(i)=F_{\cX^{+}}(i-1)\). 
While \(F_{\cX^{+}}(i)\) is an unlabelled forest, its nodes correspond to the tails in \((\cP^{+}_i
\setminus \cP^{+}_{0})\cup\cA^{+}_i\). So we can assign a label \emph{paired} or \emph{active} to each node of \(F_{\cX^{+}}(i)\).

Given half-edges \(e_1\) and \(e_2\), the distance \(\dist(e_1,e_2)\) is the length of the shortest path from \(\cV(e_1)\) to \(\cV(e_2)\) which
starts with the edge containing \(e_{1}\) and ends with the edge containing \(e_{2}\).  %For example,

If \(i_t\) is the last step where a tail at
distance \(t\) from \(\cX^{+}\) is paired, then \(F_{\cX^{+}}(i_t)\)
satisfies: (i) the height is \(t\); (ii) the set of actives nodes is the \(t\)-th level.
We call a rooted forest \(F\) \emph{incomplete} if it satisfies (i)-(ii). We let \(p(F)\) be the number of
\emph{paired} nodes in \(F\).

\subsection{Size biased distributions}

We recall some notation in \cite{cai2020a}. The \emph{in- and out-size biased} distributions of \(D_{n}\) and \(D\) are defined
\begin{align}
    \p{\dnin=(k-1,\ell)}&=\frac{k n_{k,\ell}}{m_{n}},
    \qquad
    \p{\dnout=(k,\ell-1)}=\frac{\ell n_{k,\ell}}{m_{n}},
\\
\p{{D}_{\tin}=(k-1,\ell)} &= \frac{k \lambda_{k,\ell}}{\lambda}
,
\qquad
\p{{D}_{\tout}=(k,\ell-1)} = \frac{\ell \lambda_{k,\ell}}{\lambda}
.
\end{align}
Then, by (i) of \autoref{cond:main}, 
\(\dnin \to \din \) and \(\dnout \to \dout\),
and by (iii) of \autoref{cond:main},
\begin{equation}\label{eq:dnin:nu}
    \lim_{n \to \infty}
    \E{
        \dnin^{+}
    }
    =
    \lim_{n \to \infty}
    \E{
        \dnout^{-}
    }
    =
    \E{\dinp}
    =
    \E{\doutm}
    =    
    \frac{\mathbb{E}\left[D^{+} D^{-}\right]}
    \lambda
    =
    \nu
    .
\end{equation}

Let \(s_{n+}\), \(s_{n-}\), \(s_{+}\) and \(s_{-}\) be the survival probabilities of the branching
processes with distribution \(\dninp\), \(\dnoutm\), \(D_\tin^{+}\) and \(D_\tout^{-}\) respectively.
Then as we have shown in \cite{cai2020a}, \(s_{n\pm} \to s_{\pm}\).

\subsection{Coupling with branching processes}
Consider the probability distribution \(Q_n\coloneqq(D_n)_\tin^{+}\) which satisfies for all \(\ell\geq 0\),
\begin{equation}
    \p{Q_n=\ell} = q_{n,\ell} \coloneqq \frac{\sum_{k\geq 1}k n_{k,\ell}}{m_n}.
\end{equation}
In \cite[Section~3]{cai2020a}, it has been shown that \(Q_{n} \to \dinp\) in distribution
and in expectation. In particular, by \eqref{eq:dnin:nu} \(\bE[Q_n] \to \bE[\dinp] = \nu\). Also in
\cite{cai2020a}, we showed that the exploration process starting from one tail can be approximated by a branching process with offspring distribution \(Q_{n}\). Similarly, the extended exploration process starting from \(\cX^{+}\) can be approximated by \(\abs{\cX^{+}}\) independent
branching processes with offspring distribution \(Q_{n}\).

For \(\beta\in(0,1/10)\), consider the distributions  \(Q_n^\da=Q_n^\da (\beta)\)  and  \(\Qnua = \Qnua(\beta)\) defined
by
\begin{align}
    \p{\Qnda = \ell}
    &=
    q_{n,\ell}^\da \coloneqq
    \begin{cases}
        c^{\da} q_{n,\ell} & \text{if } q_{n,\ell} \geq n^{-2\beta} \text{and }\ell\leq
        n^{\beta}
        \\
        0 & \text{otherwise}
    \end{cases}
\\
    \p{\Qnua = \ell}
    &=
    q_{n,\ell}^\ua \coloneqq
    \begin{cases}
        c^{\ua} q_{n,\ell} & \ell \ge 1
        \\
        c^{\ua} q_{n,0}+n^{-1/2+2\beta} & \ell = 0
    \end{cases}
\end{align}
where \(c^{\da}\) and \(c^{\ua}\) are normalising constants.

Let \(\GW_{\xi}^{(x)} = (\GW_{1,\xi}, \dots, \GW_{x,\xi})\) be \(x\) independent Galton-Watson trees
with offspring distribution \(\xi\).  Let \(F=(T_{1},\dots, T_{x})\) be an incomplete forest.
Let \(\GW_{\xi}^{(x)} \cong F\) denote that for every \(i \in [x]\),
\(T_{i}\) is a root subtree of \(\GW_{i, \xi}\) and all paired nodes of \(T_{i}\) have the same
degree in \(\GW_{i,\xi}\).

The following lemma is a straightforward extension of \cite[Lemma 5.3]{cai2020a} and we omit its proof:
\begin{lemma}\label{lem:eq_T_GW}
    Let \(\beta \in (0,1/10)\) and let \(H\) be a partial pairing with \(|\cV(H)|\le n^{1- 6\beta}\). Let \(\cX^+\subset \cE^+\) with \(|\cX^+|=x\).
    For every incomplete forest \(F\) with
    \(p(F)\leq n^{\beta}\), we have
    \begin{equation}
        (1+o(1)) \p{\GW_{Q_n^\da{(\beta)}}^{(x)}\cong F} 
        \leq \p{F_{\cX^{+}}(p(F)) = F \cond E_{H}} 
        \leq  (1+o(1)) \p{\GW_{Q_n^{\ua}(\beta)}^{(x)}\cong F}
        .
    \end{equation}
  %  where the implicit functions \(o(1)\) are uniform over all such \(F\) and \(H\).
\end{lemma}

\section{Expansion probability}\label{sec:tower}

Let \(\cN^\pm_{t}(\cX^{\pm})\) and \(\cN^\pm_{\le t}(\cX^{\pm})\) be the sets of heads/tails at
distance \(t\) and at most \(t\) from \(\cX^{\pm} \subseteq \cE^{\pm}\) respectively.  From now on,
let
\begin{equation}
    \omega\coloneqq \log^6 n, \qquad 
    t_{0} \coloneqq \log_{\nu} \omega
    .
\end{equation}
Let \(t_{\omega}(\cX^{\pm})\) be the \emph{expansion time} of \(\cX^\pm\) defined as
\begin{equation}
    t_{\omega}(\cX^{\pm})
    \coloneqq
    \inf
    \left\{
        t \ge 1:
        \abs{
            \cN_{t}^\pm(\cX^{\pm})
        }
        \geq \omega
    \right\}
    .
\end{equation}
For brevity, we write \(\NtoOXpm = \cup_{t = 1}^{t_{\omega}} \cN^\pm_{t}(\cX^{\pm})\). 

 Given \(H\) a partial pairing of \(\cE^\pm\) and  \(\cX^{\pm}\subseteq \cE^\pm \), we consider the following two events:
\begin{equation}\label{NVMZA}
    \begin{aligned}
        &
        A_{1}(\cX^{\pm},\varepsilon) \coloneqq [t_{\omega}(\cX^{\pm}) \le (1+\varepsilon)t_{0}].
        \\
        &
        A_{2}(\cX^{\pm}, H) \coloneqq
        \left[ 
            \NtoOXpm
            \cap
            \cF^{\pm}(H)
            =
            \emptyset
        \right]
        .
    \end{aligned}
\end{equation}
The first lemma in this section shows that the probability that both these events happen is close to
the survival probability of a branching process.

\begin{lemma}\label{lem:expand:single}
    Assume that \(\nu > 1\).  Fix \(x \in \N\), \(\varepsilon \in (0,1/2)\) and \(\gamma \in
    (0,1)\).  Then uniformly for all choices of partial pairing \(H\) and \(\cX^{\pm}\subseteq \cE^\pm\) with
    \(\abs{\cV(H)} \le n^{1-\gamma}\), \(\abs{\cX^{\pm}} = x\), as \(n \to \infty\),
    \begin{equation}\label{eq:tower:p}
        \p{
            A_{1}(\cX^{\pm}, \varepsilon)
            \cap
            A_{2}(\cX^{\pm}, H)
            \cond
            E_{H}
        }
        =
        (1+o(1))
        (1-\rho_{\pm}^{x})
        .
    \end{equation}
\end{lemma}

\begin{proof}
    Let \(\cF_{x,t,\omega}\) be the class of incomplete forests \(F\) with \(x\) trees, height \(t\) and
    such that only the last level has at least \(\omega\) nodes.  Let \(t_{1} =
    \floor{(1+\varepsilon) t_{0}}\).  For \(t \le t_{1}\) and \(F \in \cF_{x,t, \omega}\), we have
    \((t-1)\leq p(F) \le x \omega t= O(\log^{7} n)\).  Let \(\beta=\gamma/100\).  Let
    \(X_{1,t}^{{\ua}},\dots,X_{x,t}^{{\ua}}\) be the sizes of the \(t\)-th generation of \(x\) iid
    branching processes with offspring distribution \(Q_n^{\ua}(\beta)\) and let \(s_{+n}^{\ua}\) be
    the survival probability of each one. Since \(Q_{n}^{\ua} \to D^+_\tin\) in distribution, we
    have \(s_{+n}^{\ua} \to s_{+} = 1-\rho_{+} > 0\).

    Let \(T_{\omega}^{\ua}= \inf\{t\geq 1:\,\sum_{i=1}^{x} X_{i,t}^{\ua} \ge \omega\}\).  By
    \autoref{ORGNW} and \autoref{lem:eq_T_GW}, the
    LHS of \eqref{eq:tower:p} is
    \begin{equation}\label{OAAQU}
        \begin{aligned}
            \sum_{t = 1}^{t_{1}}
            \sum_{j=t-1}^{\floor{x \omega t}}
            \sum_{\substack{F\in \cF_{x,t,\omega}\\ p(F)=j}} 
            \p{
                F_{\cX^{+}}(x) = F
                \cond
                E_{H}
            }
            & \leq   
            (1+o(1))
            \sum_{t = 1}^{t_{1}}
            \sum_{j=t-1}^{\floor{x \omega t}}
            \sum_{\substack{F\in \cF_{x,t,\omega}\\ p(F)=j}} \p{\GW_{Q_n^{\ua}(\beta)}\cong F}
            \\
            &
            =
            (1+o(1))
            \p{T_{\omega}^{\ua} \le t_{1}}
            \\
            &
            =
            (1+o(1))
            (1-(1-s_{+n}^{\ua})^{x})
            \\
            &
            =
            (1+o(1))
            (1-\rho_{+}^{x})
            ,
        \end{aligned}
    \end{equation}
    where we used that \(\nu > 0\) implies \(\rho_{\pm}<1\). The lower bound follows from a similar
    argument.
\end{proof}

Our next lemma shows that when 
\(\abs{\cX^{+}} \abs{\cX^{-}}\) is small, \(\cX^{+}\) and \(\cX^{-}\) are unlikely to be too close.
We omit the proof since it follows from an easy adaptation of the proof in
\cite[Proposition~7.2]{cai2020a}.

\begin{lemma}\label{lem:dist:1}
    Assume that \(\nu > 1\).
    Fix \(\varepsilon \in (0,1/2)\) and \(\gamma \in (0,1)\).  Then uniformly for all choices of partial pairing \(H\) and \(\cX^{\pm}\subseteq \cE^\pm\)
    with \(\abs{\cV(H)} \le n^{1-\gamma}\) and
    \(\abs{\cX^{+}}\abs{\cX^{-}} \le \omega \sqrt{n}\), we have
    \begin{equation}\label{DQGPL}
        \p{
            \dist(\cX^{+}, \cX^{-}) \le \left(\frac{1}{2}-\varepsilon \right) \log_{\nu} n
            \cond
            E_{H}
        }
        =
        o(n^{-\varepsilon/2}).
    \end{equation}
\end{lemma}

The previous lemma allows us to remove \(A_{2}(\cX^{\pm}, H)\) in \autoref{lem:expand:single}.
\begin{lemma}\label{lem:expand:double}
    Assume that \(\nu > 1\).  Fix \(x^{\pm} \in \N\) and \(\varepsilon \in (0, 1/2) \).  Then
    uniformly for all choices of partial pairing \(H\) and \(\cX^{\pm}\subseteq \cE^\pm\) with
    \(\abs{\cV(H)} = o(\omega^2)\), \(\abs{\cX^{\pm}}=x^{\pm}\), we have, as \(n \to \infty\),
    \begin{align}\label{QPGVW}
        \p{
            A_{1}(
                \cX^{\pm}
                ,
                \varepsilon
            )
            \cond
            E_{H}
        }
        &=
        (1+o(1))
        (1-\rho_{\pm}^{x^{\pm}})
        ,
    \\
    \label{AWRMK}
        \p{
            A_{1}(
                \cX^{+}
                ,
                \varepsilon
            )
            \cap
            A_{1}(
                \cX^{-}
                ,
                \varepsilon
            )
            \cond
            E_{H}
        }
        &=
        (1+o(1))
        (1-\rho_{-}^{x^{-}})
        (1-\rho_{+}^{x^{+}})
        .
    \end{align}
\end{lemma}

\begin{proof}
We will prove it for \(\cX^+\); a similar argument works for \(\cX^{-}\).   Let
    \begin{equation}\label{SSPYY}
        E_{1} = A_{1}(\cX^{+} , \varepsilon),
        \quad
        E_{2} 
        =
        A_{2}(\cX^{+}, H),
        \quad
        E_{3} = A_{1}(\cX^{-} , \varepsilon).
    \end{equation}
    Note that the event \(E_{2}\) happens if and only if \(
        \dist(\cX^{+}, \cF^{+}(H)) > t_{\omega}(\cX^{+})\).

    By \autoref{lem:expand:single}, the LHS of \eqref{QPGVW} equals
    \begin{equation}\label{QLQWX}
        \begin{aligned}
            \p{E_{1} \cond E_{H}} 
            &
            = 
            \p{E_{1} \cap E_{2} \cond E_{H}}
            +
            \p{E_{1} \cap E_{2}^{c} \cond E_{H}}
            \\
            &
            =
            (1+o(1)) (1-\rho_{+}^{x^{+}})
            +
            \p{E_{1} \cap E_{2}^{c} \cond E_{H}}
            .
        \end{aligned}
    \end{equation}
    Since \(\abs{\cV(H)} = o(\omega^{2})\), by \cite[Lemma 2.2]{cai2020a} we have
    \(\abs{\cE^{+}(H)} = o(\omega \sqrt{n})\).
    By \autoref{lem:dist:1}, for \(\delta <1/2\),
    \begin{equation}\label{NRAIV}
        \begin{aligned}
        \p{E_{1} \cap E_{2}^{c} \cond E_{H}}
        &
        \le
        \p{\dist(\cX^{+}, \cF^{+}(H)) \le 4 t_{0} \cond E_{H}}
        \\
        &
        \le
        \p{\dist(\cX^{+}, \cF^{+}(H)) \le \left( \frac{1}{2} - \delta \right) \log n \cond E_{H}}
        =
        o(1)
        .
        \end{aligned}
    \end{equation}
    %This proves \eqref{QPGVW} for \(\cX^{+}\). A similar argument works for \(\cX^{-}\).

    Let \(\cH\) be the set of all possible partial pairings in \(\cN^{\leq \omega}(\cX^{+})\) such that
    \(E_{1} \cap E_{H}\) happens. Then \(H' \in \cH\) implies that \(\abs{\cV(H')},\abs{\cV(H' \cup H)} = o(\omega^{2})\). Using \autoref{lem:expand:single} again, we have
        \begin{align}\label{XOZRK}
            \p{E_{1} \cap E_{3} \cond E_{H}}
            &
            =
            \sum_{H' \in \cH}
            \p{E_{3} \cond E_{H \cup H'}}
            \p{E_{H \cup H'} \cond E_{H}}
            \nonumber
            \\
            &
            =
            \sum_{H' \in \cH}
            (1+o(1)) (1-\rho_{-}^{x^{-}}) 
            \p{E_{H' \cup H} \cond E_{H}}
            \\
            &
            =
            (1+o(1)) (1-\rho_{-}^{x^{-}}) 
            \p{E_{1}\cond E_{H}}
            \nonumber
            \\
            &
            =
            (1+o(1)) 
            (1-\rho_{-}^{x^{-}}) 
            (1-\rho_{+}^{x^{+}}) 
            .
            \qedhere
        \end{align}      
\end{proof}

Unsurprisingly, \autoref{lem:expand:double} can be extended to a fixed number of pairs of head-sets
and tail-sets:

\begin{lemma}\label{lem:expand:4}
    Assume that \(\nu > 1\).  Fix \(i, x_{1}^{\pm}, \dots, x_{i}^{\pm} \in \N\) and
    \(\varepsilon \in (0, 1/2) \).  Then uniformly for all disjoint sets of tails \((\cX_{1}^{+},
    \dots, \cX_{i}^{-})\) and disjoint sets of heads \((\cX_{1}^{+},  \dots, \cX_{i}^{-})\) with
    \(\abs{\cX_{j}^{\pm}}=x_{j}^{\pm}\) for \(j \in [i]\), we have, as \(n \to \infty\),
    \begin{equation}\label{PKURP}
        \p{
            \cap_{j=1}^{i}
            [
                A_{1}(
                    \cX^{+}
                    ,
                    \varepsilon
                )
                \cap
                A_{1}(
                    \cX^{-}
                    ,
                    \varepsilon
                )
            ]
        }
        =
        (1+o(1))
        \prod_{j=1}^{i}
        (1-\rho_{-}^{x_{j}^{-}})
        (1-\rho_{+}^{x_{j}^{+}})
        .
    \end{equation}
\end{lemma}

\begin{proof}
    We prove it by induction. The case \(i=1\) follows by
    \autoref{lem:expand:double} with \(H\) an empty pairing.

    Let \(E_{j}\) denote the event in the LHS of \eqref{PKURP}. Assume that the lemma holds
    for some \(i \ge 1\). Let \(\cH\) be the sets of all possible partial pairings in
    \(\cup_{j=1}^{i} [\cN^{\leq \omega}(\cX_{j}^{+}) \cup \cN^{\leq \omega}(\cX_{j}^{-})]\) compatible with \(E_{i}\). If \(H \in \cH\), then
    \(|\cV(H)|= o(\omega^{2})\). Using \autoref{lem:expand:single} as in~\eqref{PKURP}, we conclude
    \begin{equation*}\label{NPJVD}
            \p{E_{i+1}}
            =
            \sum_{H \in \cH}
            \p{E_{i+1} \cond E_{H}}
            \p{E_{H}}
            =
            (1+o(1)) (1-\rho_{-}^{x_{i+1}^{-}}) (1-\rho_{+}^{x_{i+1}^{+}}) 
            \p{E_{i}}
            .\qedhere
    \end{equation*}
   % where second last step uses \autoref{lem:expand:single}.
    %It then follows from the induction hypothesis that the lemma also holds for \(j+1\).
\end{proof}

The last lemma shows that expansions are unlikely to happen very late. 
\begin{lemma}\label{lem:dist:late}
    Assume that \(\nu > 1\).  Fix \(x^\pm \in \N\) and \(\varepsilon \in (0, 1/2) \).  Then uniformly
    for all choices of  \(\cX^{\pm}\subseteq \cE^\pm\) with \(\abs{\cX^{\pm}}=x^{\pm}\), as \(n\to \infty\),
    \begin{equation}\label{CDIOJ}
        \p{
        t_{\omega}(\cX^{\pm}) \in ((1+\varepsilon)t_{0}, \infty)
        }
        =
        o(1).
    \end{equation}
\end{lemma}

\begin{proof}
    Let \(t_{1} = \floor{(1+\varepsilon) t_{0}}\).  
    Note that
    \begin{equation}\label{AVNZR}
        \p{
            t_{\omega}(\cX^{\pm}) \in (t_{1}, \infty)
        }
        \le
        \sum_{e^{\pm} \in \cX^{\pm}}
        \p{
            t_{\omega}(e^{\pm}) \in (t_{1}, \infty)
        }
        .
    \end{equation}
    Thus we may assume that \(\cX^{\pm} = \{e^{\pm}\}\).  Let
    \(X_{t}^{\ua}\) be the size of the \(t\)-th generation of a branching process with offspring
    distribution \(Q_{n}^{\ua}(\beta)\) for some \(\beta \in (0,1/10)\).  Let
    \(T_{\omega}=\inf\{ t\geq 1:\,X_{t}^{\ua} \ge \omega\}\).  Then it follows from
    \cite[Theorem~3.4]{cai2020a} that there exist constants \(C > 0\) and \(\hnu \in (0,1)\) such
    that
    \begin{equation}\label{PWMIJ}
        \p{
            T_{\omega} \in (t_{1}, \infty)
        }
        \le
        \p{
            \cap_{t=0}^{t_{1}}
            [X_{t} \in (0,\omega)]
        }
        \le
        C ((1+o(1) \hnu)^{(1+\varepsilon)t_{0} - (1+o(1)) t_{0}}
        =
        o(1)
        .
    \end{equation}
    By the same argument as in \autoref{lem:expand:single}, this implies 
    \( \p{ t_{\omega}(e^{\pm}) \in (t_{1}, \infty) } = o(1) \).
\end{proof}

\section{Expectation and variance}
\label{sec:L}
\begin{lemma}\label{lem:L:mmt}
    Assume that \(\nu > 1\). 
    Let
    \begin{equation}\label{KDVVA}
        \cL
        \coloneqq
        \left\{ 
            v \in [n]
            :
            t_{\omega}(\cE^{+}(v)) < \infty, \,
            t_{\omega}(\cE^{-}(v)) < \infty
        \right\}
        .
    \end{equation}
    Then
    \begin{equation}\label{MSHJN}
        \frac{\ee{\abs{\cL}}}{n}
        \to
        \eta
        ,
        \qquad
        \frac{\ee{\abs{\cL}^{2}}}{n^{2}}
        \to
        \eta^{2}
        ,
    \end{equation}
    where \(\eta\) is defined as in~\eqref{FXADH}.
    Thus, \(\abs{\cL}/n \to \eta\) in probability.
\end{lemma}

\begin{proof}
   As \(\rho_{\pm} < 1\) and \(\sum_{i,j \ge 0} 
    \lambda_{i,j} = 1\), we have \(\eta \in (0,1)\).
      Fix \(\varepsilon \in (0,1/2)\). Define 
\begin{equation}\label{MZCUQ}
    \cL(\varepsilon)
    \coloneqq
    \left\{ 
        v \in [n]
        :
        t_{\omega}(\cE^{+}(v)) < (1+\varepsilon)t_{0}, \,
        t_{\omega}(\cE^{-}(v)) < (1+\varepsilon)t_{0}
    \right\}
    .
\end{equation}
   and note that \(\cL(\varepsilon)\subseteq \cL\). Given \(v\in [n]\) with \(i\) heads and \(j\) tails, it follows from 
    \autoref{lem:expand:4} that
    \begin{equation}\label{HLSEN}
        p_{i,j} \coloneqq \p{v \in \cL(\varepsilon)} = (1+o(1)) (1-\rho_{-}^{i})(1-\rho_{+}^{j}).
    \end{equation}
    Since there are \(n_{i,j}\) such nodes, by (i) of \autoref{cond:main},
    \begin{equation}\label{EFRQT}
        \frac{\ee{\abs{\cL(\varepsilon)}}}{n}
        =
        \sum_{i,j\ge 0}
        \frac{n_{i,j}}{n}\,
        p_{i,j}
        =
        \sum_{i,j \ge 0}
        (1+o(1))
        \lambda_{i,j} 
        (1-\rho_{-}^{i})
        (1-\rho_{+}^{j})
        \to
        \eta
        .
    \end{equation}
    To see that the sum above converges to \(\eta\), note that \(\sum_{i,j \ge 0} \frac{n_{i,j}}{n}
    = 1\) and \(p_{i,j}\leq 1\). Thus we can apply the dominated convergence theorem by considering the double sum as an
    integral over \(\dsZ_{\ge 0}^{2}\) with respect to the counting measure.
    \autoref{lem:dist:late} implies \(\p{v \in \cL \setminus \cL(\varepsilon)} = o(1)\). Thus \( \ee{\abs{\cL \setminus \cL(\varepsilon)}} = o(n)\),
    which finishes the proof for the expectation.

    Given distinct \(v_{1}, v_{2}\in [n]\) with degrees \((i_{1},j_1)\) and \((i_{2},j_2)\), again by \autoref{lem:expand:4}
    \begin{equation}\label{NBDKX}
        p_{i_1,j_1,i_2,j_2}
        \coloneqq
        \p{
            [v_{1} \in \cL(\varepsilon)] 
            \cap
            [v_{2} \in \cL(\varepsilon)] 
        } 
        = 
        (1+o(1)) 
        \prod_{r=1}^{2}
        (1-\rho_{-}^{i_{r}})(1-\rho_{+}^{j_r})
        .
    \end{equation}
    By  the same convergence argument used in~\eqref{EFRQT}, we have
    \begin{equation}\label{VSXEP}
        \frac{\ee{\abs{\cL(\varepsilon)}^{2}}}{n^2}
        =
        o(1)
        +
        \sum_{i_{1},j_1,i_2,j_2 \ge 0}        
        \frac{n_{i_1,j_1}n_{i_2,j_2}}{n^2}\,
        p_{i_1,j_1,i_2,j_2}
        \to
        \eta^{2}
        .
    \end{equation}
    As \( \ee{\abs{\cL \setminus \cL(\varepsilon)}} = o(n)\), the following concludes  the proof for the second moment:
    \begin{equation*}%\label{GXUVM}
        \ee{\abs{\cL}^{2} - \abs{\cL(\varepsilon)}^{2}} 
        \le 
        2n
        \ee{\abs{\cL\setminus\cL(\varepsilon)}} 
        =
        o(n^{2}).\qedhere
    \end{equation*}
\end{proof}

  %The next lemma shows that \(\abs{\cL_{e}}\) is also concentrated around its mean.

\begin{lemma}\label{lem:Le:mmt}
    Assume that \(\nu > 1\). Let \(\cL_{e}\) be the set of edges whose both endpoints are in \(\cL\). Then
    \begin{equation}\label{LZOCR}
        \frac{\ee{\abs{\cL_{e}}}}{n}
        \to
        \lambda s_{-} s_{+}
        ,
        \qquad
        \frac{\ee{\abs{\cL_{e}}^{2}}}{n^{2}}
        \to
        (\lambda s_{-} s_{+})^{2}
        .
    \end{equation}
    Thus \(\abs{\cL_{e}}/n \to \zeta\) in probability.
\end{lemma}

\begin{proof}
    We only sketch the proof since the argument is very similar to that of \autoref{lem:L:mmt}.

    Given \(v_{1}, v_{2}\in [n]\) with degrees \((i_{1},j_1)\) and \((i_{2},j_2)\) respectively, the
    number of edges \(X_{v_1,v_2}\) from \(v_1\) to \(v_2\) satisfies
    \(\ee{X_{v_1,v_2}}=j_{1}i_2/m_{n}\).  It follows from \autoref{lem:expand:4} that conditioning
    on \(X_{v_1,v_2}\), the probability that both \(v_{1}\) and \(v_{2}\) are in \(\cL\) converges
    to \( p_{v_{1}v_{2}}\coloneqq (1-\rho_{-}^{i_{1}}) (1-\rho_{+}^{j_{2}})\). We have
    \begin{equation*}%\label{YCIFX}
        \begin{aligned}
            \frac{ \E{\abs{\cL_{e}}}}{n}
            &=\sum_{v_1,v_2\in [n]}             
            \frac{
            \ee{X_{v_1,v_2}}\cdot(1+o(1))p_{v_{1},v_{2}}
            }{
                n
            }
            \\
            &=
            \sum_{i_{1},j_1,i_2,j_2 \ge 0}
            (1+o(1))
            \frac{
            n_{i_{1},j_1}n_{i_{2},j_2}
            }{
                n
            }
            \frac{j_1 i_2}{m_{n}}           
            (1-\rho_{-}^{i_{1}})
            (1-\rho_{+}^{j_{2}})
            \\
            &
            =
            \frac{1}{\lambda}
            \sum_{i_{1},j_1\ge 0}
            \sum_{i_{2},j_2 \ge 0}
            (1+o(1))
            \lambda_{i_{1},j_{1}}
            j_{1}
            (1-\rho_{-}^{i_{1}})\cdot 
            \lambda_{i_{2},j_{2}}
            i_{2}
            (1-\rho_{+}^{j_{2}})
            \\
            &
            \to
            \lambda 
            \left( 
                1 - \frac{1}{\lambda} \frac{\partial f}{\partial w}(\rho_{-},1)
            \right)
            \left( 
                1 - \frac{1}{\lambda} \frac{\partial f}{\partial z}(1,\rho_{+})
            \right)
            = 
            \lambda 
            s_{-} s_{+}
            .
        \end{aligned}
    \end{equation*}
    %This argument can be made precise in the same way as we did for \eqref{EFRQT}.
    The proof for the second moment is similar and we omit it.
\end{proof}

\section{Proof of \autoref{thm:giant}}
\label{sec:final}

If \(\nu>1\), it suffices to show that whp the set \(\cL\) defined in~\eqref{KDVVA} exactly coincides with the largest
\scc. Then \eqref{ONDOS} and \eqref{MYYET} in \autoref{thm:giant} follow immediately from
\autoref{lem:L:mmt} and \autoref{lem:Le:mmt}.

By \cite[Proposition 7.2]{cai2020a}, uniformly for all
\(\cX^{\pm} \subseteq \cE^{\pm}\) with \(\abs{\cX^{\pm}} \ge \omega\),
\begin{equation}\label{FYHSY}
    \p{\dist(\cX^{+}, \cX^{-}) = \infty} = o(n^{-100}),
\end{equation}
and \(\cL\) is contained in a \scc{} whp. We will show that whp there is no other vertex in it. Let
\begin{equation}\label{QEPQR}
    A_{3}({e^{\pm}}, t) = 
    \left[ 
        \cap_{r = 1}^{t}
        [
        0 < \cN_{r}(e^{\pm}) < \omega
        ]
    \right]
    .
\end{equation}
By \cite[Proposition 6.1]{cai2020a}, there exists a constant \(\hnu{}_{\pm} \in (0,1)\) such that for
\(t = \Theta(\log n)\),
\begin{equation}\label{PHMKD}
    \p{A_{3}({e^{\pm}}, t)}
    =
    \hnu_{\pm}^{(1+o(1))t}
    .
\end{equation}
Thus, letting \(t_{2}^{\pm} = \ceil{2 \log _{{1}/{\hat{\nu}_{\pm}}}(n)}\), we have
\begin{equation}\label{MADDM}
    \p{
        \cup_{e^{+} \in \cE^{+}}
        \cup_{e^{-} \in \cE^{-}}
        \left[ 
        A_{3}({e^{+}}, t_{2}^{+})
        \cup
        A_{3}({e^{-}}, t_{2}^{-})
        \right]
    }
    \le
    (
        m_{n}
        n^{-3/2}
    )^{2}
    =
    o(1)
    .
\end{equation}
Therefore, whp, each node \(v \in [n] \setminus \cL\) either can reach or can be reached from at
most \(\omega t_{2}^\pm = O(\log^{7} n)\) other nodes. This implies that whp \(\cL\) is a \scc{}
and that any other \scc{} has order \( O(\log^{7} n)\).  This concludes the proof of  the
supercritical case.

For the subcritical case, we first show the following lemma.
\begin{lemma}\label{lem:cycle}
    Assume that \(\nu < 1\). Let \(C_{n,\geq \ell}\) be the number of directed simple cycles in \(\vecGn\) of length at least \(\ell\). Then,
        \begin{equation}\label{MTEEL}
        \limsup_{n \to \infty} \ee{C_{n,\geq 1}} \le \log
		\left(
        \frac{1}{1-\nu}
        \right)
        .
    \end{equation}
    Moreover, for any \(\ell_n\to \infty\), 
    \begin{equation}\label{GHBDE}
        \limsup_{n \to \infty} \ee{C_{n,\geq \ell_n}} = 0
        .
    \end{equation}
\end{lemma}

\begin{proof}
    Let \(C_{n,k}\) be the number of directed cycles of length \(k\geq 1\). (If \(k=1\), then  \(C_{n,1}\) is the number of loops.) Let \(v\in [n]\) with degrees \((i,j)\). By \cite[Lemma~7.3]{cai2020a} the
    expected number of simple paths of length \(k\) from \(\cE^{+}(v)\) to \(\cE^{-}(v)\) is  at most \( (1+o(1)) ij\nu^{k-1}/m_{n}\).
    As each cycle of length \(k\) is counted \(k\) times, we have
    \begin{equation}\label{YOTYY}
        \ee{C_{n,k}} 
        \le
        \frac{1}{k}\sum_{i,j \ge 0}
        (1+o(1)) \frac{n_{i,j} i j\nu^{k-1}}{m_{n}}
        \to
        \frac{\nu^{k}}{k}
        .
    \end{equation}
   We conclude that
    \begin{align}\label{NLLWR}
        \limsup_{n \to \infty} \ee{C_{n,\geq 1}} 
        &=
        \limsup_{n \to \infty} \sum_{k \ge 1} \ee{C_{n,k}} 
        \le 
        \sum_{k \ge 1}
        \frac{\nu^{k} }{k}
        =
        \log
         \left(
        \frac{1}{1-\nu}
		\right)        
        ,
        \\
        \limsup_{n \to \infty} \ee{C_{n,\geq \ell}} 
        &=
        \limsup_{n \to \infty} \sum_{k \ge \ell} \ee{C_{n,k}} 
        \le 
        \sum_{k \ge \ell}
        \frac{\nu^{k}}{k}
        \leq 
        \frac{1}{\ell+1}\left(\frac{\nu}{1-\nu}\right)^\ell
        ,
    \end{align}
where the last inequality follows from the error bound on the Taylor approximation of \(\log(\frac{1}{1-\nu})\).
\end{proof}

The above lemma shows that, for any \(\ell_{n} \to \infty\), whp (i) there are at most \(\ell_{n}\) cycles
in \(\vecGn\), and (ii) all cycles have length at most \(\ell_n\). As any vertex in a \scc belongs to at
least one cycle, it follows that any \scc has order at most \(\ell_n^2\).  This finishes the proof of
the subcritical case.

\begin{rem}
    In \cite{cai2017b}, it was showed that the number of cycles outside the giant of a
    uniform random \(k\)-out digraph with \(k \ge 2\) converges to a Poisson distribution.
    We believe that similar methods can be applied to derive that the law of \(C_{n, \ge 1}\) converges to
a Poisson distribution with mean \( \log (\frac{1}{1-\nu})\).
\end{rem}

\section{Binomial Random Digraphs}\label{sec:eg}

The binomial random digraph \(\dsD_{n,p}\) is a simple digraph on \([n]\) in which each ordered pair
of nodes is connected with an arc independently at random with probability \(p\), see~\cite[Chapter
12]{frieze2015}.

Although the degrees of nodes in \(\dsD_{n,p}\) are random, conditioning on its degree sequence,
\(\dsD_{n,p}\) has the same probability to be any simple digraph with such a degree sequence. Thus
we can study its properties through the directed configuration model. Using this method, we were
able to show that the diameter of \(\dsD_{n,p}\) converges in probability in \cite[Theorem
9.5]{cai2020a}.

The same argument can be applied to determine the largest \scc in \(\dsD_{n,p}\).  Assuming that \(n
p \to \nu\), the degree of a uniform random node in \(\dsD_{n,p}\) converges in distribution to two
independent Poisson random variables with mean \(\nu\).  Thus, by \autoref{thm:giant} we recover the
following result by Karp~\cite{karp1990}:

\begin{thm}\label{thm:binom}
    Assume that \(n p \to \nu \).
    Let \(\rho\) be the smallest solution of \(\rho = 
    e^{-\nu(1-\rho)}\) on \( (0,1]\). 
    Let \(\giant{}\) be the largest \scc in \(\dsD_{n,p}\).
    Then
    \begin{equation}\label{AOPLN}
        \frac{v(\giant)}{n}
        \to
        \left( 1- \rho \right)^{2}
        ,
        \qquad
        \frac{e(\giant)}{n}
        \to
        \nu
        \left( 1- \rho \right)^{2}
        ,
    \end{equation}
    in expectation, in second moment and in probability.
\end{thm}
The case \(\nu=1\) has attracted some attention recently. Coulson~\cite{coulson2019}
determined the critical window of the model, and Goldschmidt and Stephenson~\cite{goldschmidt2019}
showed convergence of the  sequence of rescaled largest \scc within the critical window.

Pittel and Poole~\cite{pittel2016} showed that in fact the joint distribution of \(v(\giant)\) and
\(e(\giant)\) is asymptotically Gaussian in \(\dsD_{n, p}\). It is interesting to see if this
holds in the directed configuration model.

\bibliographystyle{abbrvnat}
\bibliography{rdg}

\end{document}